\documentclass[a4paper,11pt]{amsart}
\usepackage[titletoc,toc,title]{appendix}
\usepackage[left=2.7cm,right=2.7cm,top=3.5cm,bottom=3cm]{geometry}
\usepackage{amssymb,latexsym,amsmath,amsthm,amscd}
\usepackage{graphicx}
\usepackage{dsfont}
\usepackage[all]{xy}
\usepackage{mathrsfs}
\usepackage{hyperref}
\usepackage{color}
\definecolor{Blue}{rgb}{0.3,0.3,0.9}

\usepackage[T2A,OT1]{fontenc}
\DeclareSymbolFont{cyrillic}{T2A}{cmr}{m}{n}
\DeclareMathSymbol{\Sha}{\mathalpha}{cyrillic}{216}

\newtheorem{thm}{Theorem}[section]
\newtheorem{def-thm}[thm]{Definition-Theorem}
\newtheorem{cor}[thm]{Corollary}

\newtheorem{def-lem}[thm]{Definition-Lemma}
\newtheorem{prop}[thm]{Proposition}
\newtheorem{conj}[thm]{Conjecture}

\theoremstyle{definition}

\theoremstyle{remark}
\newtheorem*{rem}{Remark}
\numberwithin{thm}{section}
\numberwithin{equation}{section}

\newcommand{\cO}{\mathcal{O}}

\newcommand{\pp}{\mathfrak{p}}

\newcommand{\bQ}{\mathbf{Q}}

\newcommand{\bZ}{\mathbf{Z}}

\begin{document}

\title[Exceptional zeros for Heegner points and $p$-converse]{Exceptional zeros for Heegner points and $p$-converse to the theorem of Gross--Zagier and Kolyvagin}


\author{Francesc Castella}
\address{University of California Santa Barbara, South Hall, Santa Barbara, CA 93106, USA}
\email{castella@ucsb.edu}



\date{\today}



\maketitle

\begin{abstract}
We prove a $p$-converse to the theorem of Gross--Zagier and Kolyvagin for elliptic curves $E/\bQ$ at primes $p>3$ of multiplicative reduction. Two key ingredients in the argument are an extension to this setting of a $p$-adic formula of Bertolini--Darmon--Prasanna obtained in our earlier work \cite{cas-split}, and an exceptional zero formula for Heegner points. 

By independent approaches different from ours, a similar $p$-converse theorem was obtained by Skinner--Zhang \cite{skinner-zhang} under additional ramification hypotheses on $E[p]$, and by Venerucci \cite{venerucci-conv} assuming  finiteness of $\Sha(E/\bQ)[p^\infty]$. 
\end{abstract}

%


\section{Introduction}\label{sec:intro}

\subsection{Main results} 

The main result of this note is the following $p$-converse  to the theorem of Gross--Zagier and Kolyvagin for rational elliptic curves at primes $p$ of multiplicative reduction. 

\begin{thm}\label{thm:p-conv}
Let $E/\bQ$ be an elliptic curve with multiplicative reduction at the prime $p>3$. Assume that
\begin{itemize}
\item[(i)] $E[p]$ is irreducible as a $G_{\bQ}:={\rm Gal}(\overline{\bQ}/\bQ)$-module,
\item[(ii)] $E$ has nonsplit multiplicative reduction at some prime $q\neq p$ where $E[p]$ is ramified,  
\item[(iii)] $E(\bQ_p)[p]=0$.
\end{itemize}
Then
\[
{\rm corank}_{\bZ_p}{\rm Sel}_{p^\infty}(E/\bQ)=1\quad\Longrightarrow\quad{\rm ord}_{s=1}L(E,s)=1,
\]
and so ${\rm rank}_{\bZ}E(\bQ)=1$ and $\#\Sha(E/\bQ)<\infty$.
\end{thm}

As an important application, Theorem~\ref{thm:p-conv} allows us to extend the reach of $p$-adic conditions under which an elliptic curve $E$ over $\bQ$ can be shown to have algebraic and analytic rank one (and hence to satisfy the Birch--Swinnerton-Dyer conjecture).

\begin{cor}\label{cor:pBSD}
Let $E/\bQ$ be an elliptic curve with multiplicative reduction at the prime $p>3$ satisfying conditions (i)--(iii) in Theorem~\ref{thm:p-conv}. If the $p$-Selmer group ${\rm Sel}_p(E/\bQ)\subset{\rm H}^1(\bQ,E[p])$ has order $p$, then
\[
{\rm rank}_{\bZ}E(\bQ)={\rm ord}_{s=1}L(E,s)=1
\] 
and $\#\Sha(E/\bQ)<\infty$; 
in particular, the Birch--Swinnerton-Dyer conjecture holds for $E$.
\end{cor}

\begin{proof}
This is immediate from Theorem~\ref{thm:p-conv} and the existence of the alternating Cassels--Tate pairing on the quotient of ${\rm Sel}_{p^\infty}(E/\bQ)$ by its maximal divisible subgroup (see \cite[Thm.~9]{majority} for details, and note that the argument also shows $\Sha(E/\bQ)[p^\infty]=0$).
\end{proof}

\begin{rem}
Conversely, under the hypotheses of Corollary~\ref{cor:pBSD} (in fact, just assuming $E(\bQ)[p]=0$), condition ${\rm Sel}_p(E/\bQ)\simeq\bZ/p\bZ$ is necessary for the conclusion 
of that result to hold, as one sees immediately from the $p$-descent exact sequence
\[
0\rightarrow E(\bQ)/pE(\bQ)\rightarrow{\rm Sel}_p(E/\bQ)\rightarrow\Sha(E/\bQ)[p]\rightarrow 0.
\]
\end{rem}

We deduce Theorem~\ref{thm:p-conv} from the proof of a new variant of Perrin-Riou's Heegner point main conjecture \cite{PR-HP} reflecting the presence of ``exceptional zeros'' (in the sense of \cite{MTT}) in the split multiplicative case. We further explain this in the following.

Let $E$ and $p$ be as in Theorem~\ref{thm:p-conv}, let $N$ be the conductor of $E$, and let $K$ be an imaginary quadratic field of odd discriminant $-D_K<-4$ and satisfying the classical \emph{Heegner hypothesis} relative to $N$:
\begin{equation}\label{eq:Heeg-main}
\textrm{there exists an ideal $\mathfrak{N}\subset\cO_K$ with $\cO_K/\mathfrak{N}\simeq\bZ/N\bZ$}.\tag{Heeg}
\end{equation}

Let $K_\infty$ be the anticyclotomic $\bZ_p$-extension of $K$, let $L$ run over the finite extensions of $K$ contained in $K_\infty$, and put
\[
\check{S}_p:=\varprojlim_L\varprojlim_r{\rm Sel}_{p^r}(E/L),
\]
where the limits are with respect to the corestriction and the multiplication-by-$p$ maps. Upon the choice of a modular parametrization $X_0(N)\rightarrow E$, whose existence is granted by \cite{BCDT} and which we fix from now on, CM points on $X_0(N)$ give rise a collection of Heegner points $y_c\in E(H_c)$ defined over the ring class field $H_c$ of $K$ of conductor $c$ prime to $N$. Here we are interested in the case $c=p^m$ for arbitrary $m\geq 1$ (so $(c,N)>1$); as explained in \cite[\S{2}]{BDmumford-tate}, the construction of $y_c$ extends to this case.

Let $\Lambda=\bZ_p[\![{\rm Gal}(K_\infty/K)]\!]$ be the anticyclotomic Iwasawa algebra.
From the norm relations satisfied by the points $y_{c}$, taking the images under the Kummer map of (a slight modification of) the points $y_{p^m}$ for varying $m\geq 1$, we obtain a $\Lambda$-adic class
\[
\mathbf{z}_\infty\in\check{S}_p.
\]
These same norm relations imply, in the case where $p$ is a prime of \emph{split} multiplicative reduction for $E$, that $\mathbf{z}_\infty$ vanishes under the natural map (corestriction) 
\[
{\rm pr}_0:\check{S}_p\rightarrow\varprojlim_r{\rm Sel}_{p^r}(E/K).
\]
We can then show that existence of a ``derived'' Heegner class $\mathbf{z}_\infty'\in\check{S}_p$ defined by the relation $\mathbf{z}_\infty=(\gamma-1)\mathbf{z}_\infty'$, where $\gamma\in{\rm Gal}(K_\infty/K)$ is a topological generator. 
Put
\[
{\rm Sel}_{p^\infty}(E/K_\infty)=\varinjlim_L{\rm Sel}_{p^\infty}(E/L),
\]
where the limit is with respect to the restriction maps,  
and let 
\[
X={\rm Hom}_{\rm cts}({\rm Sel}_{p^\infty}(E/K_\infty),\bQ_p/\bZ_p)
\]
be its Pontryagin dual. Finally, let
\[
\mathbf{z}_\infty^{*}\in\check{S}_p
\]
stand for the $\Lambda$-adic Heegner class $\mathbf{z}_\infty$ or its derivative $\mathbf{z}_\infty'$ according to whether $E$ has nonsplit or split multiplicative reduction at $p$, respectively. 

The key new result we obtain 
in this note, from which the proof of Theorem~\ref{thm:p-conv} is ultimately deduced, is the following.

\begin{thm}\label{thm:HPMC}
Let $E/\bQ$ be an elliptic curve with multiplicative reduction at the prime $p>3$, and let $K$ be an imaginary quadratic field of odd discriminant $-D_K<-4$ satisfying hypothesis \eqref{eq:Heeg-main} relative to the conductor $N$ of $E$. Assume that:
\begin{itemize}
\item[(i)] 
$E[p]$ is irreducible as a $G_\bQ$-module, 
\item[(ii)] If $2$ is nonsplit in $K$, then $2\Vert N$,
\item[(iii)] $E$ has nonsplit multiplicative reduction at each prime $q\Vert N$ which is nonsplit in $K$, and that there is at least one such prime $q$ at which $E[p]$ is ramified, 
\item[(iv)] $E(\bQ_p)[p]=0$.
\end{itemize}
Then $X$ and $\check{S}_p$ both have $\Lambda$-rank one with
\[
{\rm char}_\Lambda(X_{\rm tors})={\rm char}_\Lambda\bigl(\check{S}_p/\Lambda\mathbf{z}_\infty^{*}\bigr)^2,
\]
where the subscript ${\rm tors}$ denotes the maximal $\Lambda$-torsion submodule.
\end{thm}

\begin{rem}
The conclusion of Theorem~\ref{thm:HPMC} might be viewed as  an extension of the Heegner point main conjecture in  \cite[Conj.~B]{PR-HP} to the multiplicative case, 
in a ``primitive'' form reflected in the appearance of the derivative class $\mathbf{z}_\infty'$ accounting for an exceptional zero phenomenon in the case of split multiplicative reduction (see Remark after Proposition~\ref{prop:equiv} for further details). 
\end{rem}

The proof of Theorem~\ref{thm:HPMC} is based on an extension of the ``explicit reciprocity law'' of \cite{cas-hsieh1} to the multiplicative case, which we obtain in Theorem~\ref{thm:ERL} and Corollary~\ref{cor:ERL} building on the $p$-adic Abel--Jacobi image computations in our earlier work \cite{cas-split}. The formula we obtain relates the image of $\mathbf{z}_{\infty}^{*}$ under an anticyclotomic variant of Perrin-Riou's big logarithm map (or a generalized Coleman power series map) to the $p$-adic Rankin $L$-function $L_\pp(f)$ of Bertolini--Darmon--Prasanna \cite{bdp}, as extended in \cite{hsieh,cas-split} to the $p$-multiplicative case. 
With this link between $\mathbf{z}^{*}_\infty$ and $L_\pp(f)$ in hand, we deduce Theorem~\ref{thm:HPMC} from the cases of the Iwasawa Main Conjecture for $L_\pp(f)$ proved in \cite{cas-CJM} and \cite{cas-CJM-err}. 

Finally, the proof of Theorem~\ref{thm:p-conv} is deduced by an application of Theorem~\ref{thm:HPMC} for a carefully chosen $K$ together with a variant of Mazur's control theorem.

\subsection{Relation to previous works}

Earlier results on the $p$-converse to the theorem of Gross--Zagier and Kolyvagin for multiplicative primes are due to Skinner--Zhang \cite{skinner-zhang} and independently Venerucci \cite{venerucci-conv}, both obtained by methods markedly different from ours. The result of \cite{skinner-zhang} is deduced as a consequence of their proof of Kolyvagin's nonvanishing conjecture for multiplicative primes $p$, and it requires additional ramification hypotheses on $E[p]$ arising from the proof of that result; while the main result of \cite{venerucci-conv} is subject to the hypothesis that $\#\Sha(E/\bQ)[p^\infty]<\infty$ (but, compared to  Theorem~\ref{thm:p-conv}, allows both $E(\bQ_p)[p]\neq 0$, and $E$ to have split multiplicative reduction at the prime $q\Vert N$, with $q\neq p$, where $E[p]$ is required to  ramify).

We also refer the reader to \cite{buy-CMH}, \cite{rivero-rotger,rivero-forum}, and \cite{buy-mtt,venerucci-invmath} for other studies of ``derivative'' classes in the presence of exceptional zeros (for circular units, Rankin--Eisenstein classes, and Beilinson--Kato elements, respectively). 
In the spirit of in this note, it would be interesting to study the interplay between the results in those settings and the corresponding Iwasawa Main Conjectures.

\subsection{Acknowledgements}
During the work on this note, the author was partially supported by the NSF grant DMS-2401321 and the 2024-2025 AMS Centennial Research Fellowship.

\section{Exceptional zeros for Heegner points}\label{sec:p-BSD}\label{sec:previous}

In this section we study the behavior of $\Lambda$-adic systems of Heegner classes in the presence of ``exceptional zeros''. The key new ingredient is the explicit reciprocity law of Theorem~\ref{thm:ERL} in the split multiplicative case, involving a certain ``derived'' $\Lambda$-adic Heegner class.  

The results in this section apply more generally to newforms of weight $2$, but for simplicity we shall restrict ourselves the case of rational elliptic curves.

\subsection{Setting}\label{subsec:setting} 

Let $E/\bQ$ be an elliptic curve of conductor $N$, let $f\in S_2(\Gamma_0(N))$ be the newform associated to $E$, and let $p$ be a prime of multiplicative reduction for $E$, so $p\Vert N$. Throughout we assume that $p>3$, but it should be noted that many of our arguments apply to any odd prime $p$.

Let $K$ be an imaginary quadratic field of odd discriminant $-D_K<-4$ and ring of integers $\cO_K$ such that \eqref{eq:Heeg-main} holds.
In particular, the prime $p$ splits in $K$, say
\[
p=\pp\overline{\pp},
\]
with $\pp$ the prime of $K$ above $p$ induced by a fixed embedding $\imath_p:\overline{\bQ}\hookrightarrow\overline{\bQ}_p$. Fix once and for all an ideal $\mathfrak{N}$ as in \eqref{eq:Heeg-main} with $\pp\mid\mathfrak{N}$.

\subsection{Derived Heegner points}

Let $K_\infty$ be the anticyclotomic $\bZ_p$-extension of $K$, and put 
\[
\Gamma={\rm Gal}(K_\infty/K)\simeq\bZ_p,\quad\quad\Lambda=\bZ_p[\![\Gamma]\!].
\]

Let $\alpha$ be the $p$-th Fourier coefficient of $f$ (so $\alpha=1$ or $-1$ according to whether  $E$ has split or nonsplit multiplicative reduction at $p$). For each $m$, let  $H_{p^m}$ denote the ring class field of $K$ of conductor $p^m$, and put $H_{p^\infty}=\bigcup_{m>0} H_{p^m}$, which contains $K_\infty$ with finite index. 

Following  \cite[\S{2.5}]{BDmumford-tate}, define the \emph{regularized Heegner point} $z_m\in E(H_{p^m})\otimes\bZ_p$ by
\begin{equation}\label{eq:reg-Heeg}
z_m:=\begin{cases}
\alpha^{-m}\cdot y_{p^m}&\textrm{if $m>0$,}\\[0.2em]
u^{-1}(1-\alpha^{-1}\sigma)\cdot y_1&\textrm{if $m=0$,}
\end{cases}
\end{equation}
where $u:=\frac{1}{2}\#\cO_K^\times$ and $\sigma\in{\rm Gal}(H_1/K)$. A direct calculation using the norm relations satisfied by the points $y_m$ (see [\emph{op.\,cit.}, \S{2.4}]) shows that the points $z_m$ are norm-compatible. Thus taking their images under the Kummer map $E(H_{p^m})\otimes\bZ_p\rightarrow{\rm Sel}(H_m,T)$, where $T=\varprojlim_rE[p^r]$ denotes the $p$-adic Tate module of $E$, we obtain  the compatible family of cohomology classes
\[
\mathbf{z}_\infty:=\{z_m\}_m\in\varprojlim_m{\rm Sel}(H_{p^m},T).
\]
With a slight abuse of notation, we continue to denote by $\mathbf{z}_\infty$ its natural image in $\check{S}_p$.  On the other hand, put
\[
\kappa_f:={\rm cor}_{H_1/K}(y_0)\in{\rm Sel}(K,T),
\]
and note that the Gross--Zagier formula \cite{GZ} yields the equivalence
\begin{equation}\label{eq:GZ-cor}
\kappa_f\not\in{\rm Sel}(K,T)_{\rm tors}\quad\Longleftrightarrow\quad{\rm ord}_{s=1}L(E/K,s)=1,
\end{equation}
where ${\rm Sel}(K,T)_{\rm tors}$ denotes the torsion submodule of ${\rm Sel}(K,T)$.

Given a class $\mathfrak{Z}_\infty$ in the kernel of the specialization map
\[
{\rm pr}_0:\check{S}_p\rightarrow{\rm Sel}(K,T),
\]
we refer the reader to \cite[Lem.~3.10]{cas-split} for the definition of the ``derivative'' class $\mathfrak{Z}_\infty'\in\check{S}_p$ satisfying the relation
\begin{equation}\label{eq:deriv}
\mathfrak{Z}_\infty=(\gamma-1)\mathfrak{Z}_\infty'
\end{equation}
for $\gamma\in\Gamma$ any topological generator. When $E(K)[p]=0$ (so the $\Lambda$-module $\check{S}_p$ is torsion-free), $\mathfrak{Z}_\infty'$ is uniquely determined by \eqref{eq:deriv}. 
We also note that letting $\kappa:\Gamma\mapsto 1+p\bZ_p$ be any fixed isomorphism, the (integral) normalization $p{\rm log}_p(\kappa(\gamma))^{-1}\mathfrak{Z}_\infty'$ is independent of $\gamma$.

A key role in this note will be played by the following ``exceptional zero formula'' for $\mathbf{z}_\infty$.

\begin{thm}\label{thm:exc-zero}
Suppose $E$ has split multiplicative reduction at $p$, and that $E(K)[p]=0$. Then ${\rm pr}_0(\mathbf{z}_\infty)=0$, 
and its derivative $\mathbf{z}_\infty'\in\check{S}_p$ is such that
\[
{\rm pr}_0(\mathbf{z}_\infty')=\mathscr{L}_\pp(f,K)\cdot\kappa_f
\]
in ${\rm Sel}(K,T[1/p])$, where $\mathscr{L}_\pp(f,K)\in\bQ_p$.
\end{thm}

\begin{proof}
By definition \eqref{eq:reg-Heeg}, we see that
\begin{align*}
{\rm pr}_0(\mathbf{z}_\infty)={\rm cor}_{H_1/K}(z_0)
&=u^{-1}\sum_{\tau}y_1^\tau-u^{-1}\alpha^{-1}\sum_{\tau}y_1^{\sigma\tau}\\
&=u^{-1}(1-\alpha^{-1})\sum_{\tau}y_1^\tau,
\end{align*}
where $\tau$ runs over the elements in ${\rm Gal}(H_1/K)$. Since $\alpha=1$, this shows the first assertion. 

The second assertion follows from a result of S.\,Molina \cite[Thm.~7.5]{molina-TAMS} (see also Theorem~4.2.5.4 in \cite{disegni-kyoto} for a reformulation closest to ours). 
\end{proof}

\begin{rem}
The constant $\mathscr{L}_\pp(f,K)$ appearing in Theorem~\ref{thm:exc-zero} is a natural \emph{$\mathscr{L}$-invariant} attached to $E$ and the anticyclotomic $\bZ_p$-extension $K_\infty/K$ (see \cite[(3.9)]{cas-split}). 
Moreover, as shown in \cite{disegni-kyoto}, Theorem~\ref{thm:exc-zero} can be recast as a special case of the conjectures of Birch and Swinnerton-Dyer type formulated by Bertolini--Darmon \cite{BDmumford-tate} in the case of rank one.
\end{rem}

\subsection{Explicit reciprocity law}\label{subsec:ERL} 

For $g\in S_{2r}(\Gamma_0(N_g))$ a $p$-ordinary newform of weight $2r\geq 2$ and level $N_g$ \emph{coprime to $p$}, one of the main results of \cite{cas-hsieh1} is an ``explicit reciprocity law'' expressing the $p$-adic $L$-function $L_\pp(g)$ of Bertolini--Darmon--Prasanna \cite{bdp} as the image of a $\Lambda$-adic class 
(constructed from the $p$-adic \'{e}tale Abel--Jacobi image  of generalized Heegner cycles) under a Perrin-Riou big logarithm map. An extension to the multiplicative case of the $p$-adic Abel--Jacobi image computations of \cite{bdp} was obtained in \cite{cas-split}; in this section we deduce from this a corresponding extension of the aforementioned result from \cite{cas-hsieh1}.
%

Put
\[
\mathbf{T}:=T\hat\otimes_{\bZ_p}\Lambda
\]
equipped with the diagonal $G_K$-action, where $G_K$ acts on $\Lambda$ through the inverse of the character  $G_K\twoheadrightarrow\Gamma\hookrightarrow\Lambda^\times$. By Tate's uniformization, the Tate module $T$ fits into a $G_{\bQ_p}$-equivariant short exact sequence
\[
0\rightarrow\mathscr{F}^+T\rightarrow T\rightarrow\mathscr{F}^-T\rightarrow 0
\]
with $\mathscr{F}^\pm T$ free of rank $1$ over $\bZ_p$, and with the $G_{\bQ_p}$-action on $\mathscr{F}^-T$ given by the unramified character ${\rm unr}(\alpha)$ sending a Frobenius element ${\rm Fr}_p\in G_{\bQ_p}/I_p$ to $\alpha$ (and so $G_{\bQ_p}$ acts on $\mathscr{F}^+T$ via ${\rm unr}(\alpha)^{-1}\varepsilon_{\rm cyc}$ for the $p$-adic cyclotomic character $\varepsilon_{\rm cyc}$.). Put also
\[
\mathscr{F}^\pm\mathbf{T}:=\mathscr{F}^\pm T\hat\otimes_{\bZ_p}\Lambda.
\]
Let $R_0=\hat\bZ_p^{\rm nr}$ denote the completion of the ring of integers of the maximal unramified extension of $\bQ_p$, set $\Lambda_{R_0}=\Lambda\hat\otimes_{\bZ_p}R_0$, and denote by
\[
L_\pp(f)\in\Lambda_{R_0}
\]
the $p$-adic Rankin $L$-function of Bertolini--Darmon--Prasanna \cite{bdp}, as extended in \cite{hsieh,cas-split} to the $p$-multiplicative case.


\subsubsection{Split multiplicative case}

It follows from its geometric construction, that 
the image of $\mathbf{z}_\infty$ under the restriction 
\begin{equation}\label{eq:loc-pp}
{\rm loc}_\pp:\check{S}_p\rightarrow{\rm H}^1(K_\pp,\mathbf{T})
\end{equation}
lands in the image of the natural map ${\rm H}^1(K_\pp,\mathscr{F}^+\mathbf{T})\rightarrow{\rm H}^1(K_\pp,\mathbf{T})$ (cf. \cite[p.\,604]{cas-hsieh1}), which is easily seen to be an injection (see e.g. \cite[Lem.~2.4.4]{howard-invmath}). In the following, we shall often write $\bQ_p$ and $K_\pp$ interchangeably.


The next is the key new result in this note.

\begin{thm}\label{thm:ERL}
Assume that $E$ has split multiplicative reduction at $p$. Then there is a $\Lambda$-linear isomorphism
\[
\tilde{\mathcal{L}}_\pp:{\rm H}^1(K_\pp,\mathscr{F}^+\mathbf{T})\hat\otimes_{\bZ_p}R_0\rightarrow\Lambda_{R_0}
\]
such that
\[
\tilde{\mathcal{L}}_\pp({\rm loc}_\pp(\mathbf{z}_\infty'))=L_\pp(f)
\]
up to a $p$-adic unit.
\end{thm}

\begin{proof}
For a (typically infinite) Galois extension $E_\infty$ of $\bQ_p$ with Galois group $\mathcal{G}={\rm Gal}(E_\infty/\bQ_p)$, we denote by $\Lambda_{\mathcal{G}}^\#$ the Iwasawa algebra $\bZ_p[\![\mathcal{G}]\!]$ equipped with the Galois action by the inverse of the tautological character 
\[
G_{\bQ_p}\twoheadrightarrow\mathcal{G}\hookrightarrow\Lambda_{\mathcal{G}}^\times;
\]
and for any abelian extension $F$ of $K$, let $F_\pp$ denote the completion of $F$ at the prime above $\pp$ induced by our fixed embedding $\imath_p:\overline{\bQ}\hookrightarrow\overline{\bQ}_p$. 

It follows easily from local class field theory, that $K_{\infty,\pp}$ is contained in the $\bZ_p^\times$-extension $F_\infty=\bigcup_{n\geq 0}F(\mathcal{F}[p^n])$ of $F:=H_{1,\pp}$ associated to a height $1$ Lubin--Tate formal group $\mathcal{F}$ over $\cO_F$ (see e.g.  \cite[Prop.~39]{shnidman}). We shall deduce the construction of $\tilde{\mathcal{L}}_\pp$ from the theory of Coleman power series \cite{coleman-division,de_shalit} for $F_\infty/F$. (Alternatively, it could also be deduced from a suitable specialization of \cite[Prop.~5.2]{cas-variation} under some hypotheses.) 

Put 
\[
\mathcal{G}={\rm Gal}(F_\infty/F),\quad\quad\tilde{\mathcal{G}}={\rm Gal}(F_\infty/\bQ_p).
\]
As in \cite[Def.~8.2.1]{KLZ2}, for $\mathcal{T}$ an unramified and $p$-adically complete $\bZ_p[G_{\bQ_p}]$-module, put
\[
\mathbf{D}(\mathcal{T}):=(\mathcal{T}\hat\otimes_{\bZ_p}R_0)^{G_{\bQ_p}}.
\]
Let $\hat{\bQ}_p^{\rm ur}=R_0[1/p]$ be the completion of the maximal unramified extension of $\bQ_p$, 
and let $U_{\infty}$ denote the $p$-completion of the group of norm-compatible sequences of units in the $\bZ_p^\times$-tower $\hat{\bQ}_p^{\rm ur}(\mathcal{F}[p^\infty])/\hat{\bQ}_p^{\rm ur}$. As explained in \cite[\S{4.1}]{FK-sharifi} and \cite[\S{4.1}]{sharifi-JEMS}, the classical Coleman power series map
\[
{\rm Col}:U_{\infty}\rightarrow\Lambda_{\mathcal{G}}\hat\otimes_{\bZ_p}R_0
\]
(as constructed \cite[\S{I.3}]{de_shalit} for general height $1$ Lubin--Tate formal groups, generalizing the basic case of  the multiplicative group $\hat{\mathbf{G}}_m$) extends uniquely to a $\Lambda_{R_0}$-linear map
\begin{equation}\label{eq:Col}
{\rm Col}_{\mathcal{T}}:{\rm H}^1(\bQ_p,\mathcal{T}(1)\hat\otimes_{\bZ_p}\Lambda_{\tilde{\mathcal{G}}}^\#)\hat\otimes_{\bZ_p}R_0\rightarrow S^{-1}\Lambda_{\tilde{\mathcal{G}}}\hat\otimes_{\bZ_p}\mathbf{D}(\mathcal{T})\hat\otimes_{\bZ_p}R_0,
\end{equation}
where $S$ denotes the multiplicative subset of $\Lambda_{\tilde{\mathcal{G}}}$ generated by the elements in the kernel of the augmentation map $\Lambda_{\tilde{\mathcal{G}}}\rightarrow\bZ_p$. Taking 
\begin{equation}\label{eq:T}
\mathcal{T}=\mathscr{F}^+T(-1),
\end{equation} 
as we shall do from now on, the module $\mathbf{D}(\mathcal{T})$ defines a $\bZ_p$-lattice in $\mathbf{D}_{\rm crys}(\mathscr{F}^+T[1/p])$ (see \cite[Rem.~8.2.2]{KLZ2}), and hence pairing against the differential $\omega_f$ attached to $f$ we obtain a $\bZ_p$-linear isomorphism
\begin{equation}\label{eq:omega_f}
\omega_f:\mathbf{D}(\mathcal{T})\rightarrow\bZ_p
\end{equation} 
(cf. \cite[Prop.~10.1.1]{KLZ2}). Put $\Gamma_\pp={\rm Gal}(K_{\infty,\pp}/K_\pp)$, and note that 
\begin{equation}\label{eq:T(1)}
\mathcal{T}(1)\simeq\bZ_p({\rm unr}(\alpha)).
\end{equation} 
Hence from local Tate duality we see that corestriction for $F_\infty/K_{\infty,\pp}$ defines an isomorphism
\[
{\rm H}^1(\bQ_p,\mathcal{T}(1)\hat\otimes_{\bZ_p}\Lambda_{\tilde{\mathcal{G}}}^\#)\otimes_{\Lambda_{\tilde{\mathcal{G}}}}\Lambda_{\Gamma_\pp}\simeq
{\rm H}^1(\bQ_p,\mathcal{T}(1)\hat\otimes_{\bZ_p}\Lambda_{\Gamma_\pp}^\#),
\]
where the second tensor product on the left is via the natural projection $\Lambda_{\tilde{\mathcal{G}}}\rightarrow\Lambda_{\Gamma_\pp}$. Hence composing \eqref{eq:Col} and \eqref{eq:omega_f} and corestricting to $K_{\infty,\pp}$ we obtain the $\Lambda_{\Gamma_\pp}$-linear map
\begin{equation}\label{eq:Col-p}
\bigl\langle{\rm Col}_{\mathcal{T}}(-),\omega_f\bigr\rangle:{\rm H}^1(\bQ_p,\mathcal{T}(1)\hat\otimes_{\bZ_p}\Lambda_{\Gamma_\pp}^\#)\hat\otimes_{\bZ_p}R_0\rightarrow(\gamma_\pp-1)^{-1}\Lambda_{\Gamma_\pp}\hat\otimes_{\bZ_p}R_0,
\end{equation}
where $\gamma_\pp\in\Gamma_{\pp}$ is any topological generator; and tensoring \eqref{eq:Col-p} with $\Lambda$ over $\Lambda_{\Gamma_\pp}$ (via the natural algebra map $\Lambda_{\Gamma_\pp}\rightarrow\Lambda$ arising from the inclusion $\Gamma_\pp\subset\Gamma$ as a decomposition group at $\pp$) we finally obtain the $\Lambda$-linear map
\begin{equation}\label{eq:ac-log}
\mathcal{L}_\pp^{(1)}:=
\bigl\langle{\rm Col}_{\mathcal{T}}(-),\omega_f\bigr\rangle\otimes 1:{\rm H}^1(K_\pp,\mathscr{F}^+\mathbf{T})\hat\otimes_{\bZ_p}R_0\rightarrow(\gamma-1)^{-1}\Lambda_{R_0},
\end{equation} 
noting the isomorphism $\mathcal{T}(1)\hat\otimes_{\bZ_p}\Lambda_{\Gamma_\pp}^\#\otimes\Lambda\simeq\mathscr{F}^+\mathbf{T}$ as $G_{\bQ_p}=G_{K_\pp}$-modules. 

It follows from the same argument as in \cite[Lem.~4.1.3]{sharifi-JEMS} that $\mathcal{L}_\pp$ is injective, with 
\[
{\rm image}(\mathcal{L}_\pp^{(1)})=\begin{cases}
(\gamma-1)^{-1}\Lambda_{R_0}&\textrm{if $\alpha=1$,}\\[0.2em]
\Lambda_{R_0}&\textrm{if $\alpha=-1$}.
\end{cases}
\]
Moreover, by construction for any $\mathbf{z}\in{\rm H}^1(K_\pp,\mathscr{F}^+\mathbf{T})$ the Laurent series $\mathcal{L}_\pp^{(1)}(\mathbf{z})\in(\gamma-1)^{-1}\Lambda_{R_0}$ interpolates (a branch determined by $\omega_f$ of) the Bloch--Kato logarithm and dual exponential maps over specializations of $\mathbf{z}$ by characters of $\Gamma$ (cf. \cite[Thm.~3.4]{cas-hsieh-sigma}; for $\epsilon$ the $\bZ_p$-basis of the $p$-adic Tate module of $\mathcal{F}$ determined by a fixed choice $(\zeta_{p^n})_{n\geq 1}$ of compatible $p$-power roots of unity, $\mathcal{L}_\pp^{(1)}$ is essentially the inverse of the map $\Omega_{V,1}^\epsilon$ arising from \cite[Thm.~3.2]{cas-hsieh-sigma} for $V=\mathscr{F}^+T[1/p]$ in the notations in \emph{op.\,cit.}).
%

Since $E$ has multiplicative reduction at $p$ by assumption, $\alpha=1$. By the same calculation as in \cite[Thm.~5.7]{cas-hsieh1}, but replacing the appeal in [\emph{op.\,cit.}, p.\,599] to results from \cite{bdp} by their corresponding extension in \cite{cas-split} to the multiplicative case, we see that \eqref{eq:ac-log} sends the localized $\Lambda$-adic Heegner class ${\rm loc}_\pp(\mathbf{z}_\infty)$ to $L_\pp(f)$ up to a $p$-adic unit. Hence the $\Lambda$-linear map 
\[
\mathcal{L}_\pp^{(2)}:=(\gamma-1)\cdot\mathcal{L}^{(1)}_\pp
\]
is injective with image equal to $\Lambda_{R_0}$, and satisfies
\begin{equation}\label{eq:ERL-sp0}
\mathcal{L}_\pp^{(2)}({\rm loc}_\pp(\mathbf{z}_\infty))=(\gamma-1)\cdot L_\pp(f)
\end{equation}
up to a $p$-adic unit. 
Since $\mathbf{z}_\infty=(\gamma-1)\cdot\mathbf{z}_\infty'$ by Theorem~\ref{thm:exc-zero} with ${\rm loc}_\pp(\mathbf{z}_\infty')\in{\rm H}^1(K_\pp,\mathscr{F}^+\mathbf{T})$, restricting $\mathcal{L}_\pp^{(2)}$ to the subspace
\[
(\gamma-1){\rm H}^1(K_\pp,\mathscr{F}^+\mathbf{T})\subset{\rm H}^1(K_\pp,\mathscr{F}^+\mathbf{T})
\]
and dividing the result by $(\gamma-1)$ (noting that ${\rm H}^1(K_\pp,\mathscr{F}^+\mathbf{T})$ is torsion-free, as follows from the fact that ${\rm H}^0(K_{\infty,\pp},\mathscr{F}^+T)=0$), we obtain a map $\tilde{\mathcal{L}}_\pp$ as in the statement which by \eqref{eq:ERL-sp0} satisfies the stated explicit reciprocity law. 
\end{proof}

\subsubsection{Nonsplit multiplicative case}

We record the version of Theorem~\ref{thm:ERL} in the case where $E$ has nonsplit multiplicative reduction at $p$ obtained in the course of its proof.

\begin{cor}\label{cor:ERL}
Assume that $E$ has nonsplit multiplicative reduction at $p$.  Then there exists a $\Lambda$-linear isomorphism
\[
\mathcal{L}_\pp:{\rm H}^1(K_\pp,\mathscr{F}^+\mathbf{T})\hat\otimes_{\bZ_p}R_0\rightarrow\Lambda_{R_0}
\]
such that
\[
\mathcal{L}_\pp({\rm loc}_\pp(\mathbf{z}_\infty))=L_\pp(f)
\]
up to a $p$-adic unit.
\end{cor}

\begin{proof}
After the proof of Theorem~\ref{thm:ERL}, it suffices to note that when $E$ has nonsplit multiplicative reduction of $p$, so $\alpha=-1$, the map \eqref{eq:ac-log} is $\Lambda_{R_0}$-valued, so in this case directly from the calculations in \cite{cas-hsieh1} and \cite{cas-split} we obtain the result.
\end{proof}

\section{Iwasawa theory}\label{sec:p-conv}

We keep the setting of \S\ref{subsec:setting}, and let $\mathbf{z}_\infty^{*}\in\check{S}_p$ be the $\Lambda$-adic Heegner class defined by
\[
\mathbf{z}_\infty^{*}:=\begin{cases}
\mathbf{z}_\infty'&\textrm{if $E$ has split multiplicative reduction at $p$,}\\[0.2em]
\mathbf{z}_\infty&\textrm{otherwise.}
\end{cases}
\]
In this section we describe the two ingredients we need from Iwasawa theory: some results on variants of the anticyclotomic Iwasawa Main Conjecture for $E/K$ (one in terms of $L_\pp(f)$, and another in terms  $\mathbf{z}_\infty^{*}$, which is new to this paper),  
and a variant of Mazur's control theorem.

\subsection{Anticyclotomic Main Conjectures}

Similarly as in \cite[\S{2.3}]{JSW} and \cite[\S{2.1}]{cas-split}, define the anticyclotomic Selmer group 
\[
{\rm Sel}_\pp(K_\infty,E[p^\infty]):={\rm ker}\biggl\{{\rm H}^1(K_\infty,E[p^\infty])\rightarrow\prod_{w}\frac{{\rm H}^1(K_{\infty,w},E[p^\infty])}{{\rm H}_\pp^1(K_{\infty,w},E[p^\infty])}\biggr\},
\]
where $w$ runs over all finite primes of $K_\infty$, and we put
\[
{\rm H}_\pp^1(K_{\infty,w},E[p^\infty]):=
\begin{cases}
{\rm H}^1(K_{\infty,w},E[p^\infty])&\textrm{if $w\mid\overline{\pp}$,}\\[0.2em]
0&\textrm{otherwise.}
\end{cases}
\]
Write $X_\pp$ 
for the Pontryagin dual of ${\rm Sel}_\pp(K_\infty,E[p^\infty])$. The following is one of the main results of \cite{cas-CJM}, in the final form given in \cite{cas-CJM-err}. 

\begin{thm}\label{thm:BDP-IMC}
Assume that:
\begin{itemize}
\item[(i)] 
$E[p]$ is irreducible as a $G_\bQ$-module,
\item[(ii)] If $2$ is nonsplit in $K$, then $2\Vert N$,
\item[(iii)] $E$ has nonsplit multiplicative reduction at each prime $q\Vert N$ which is nonsplit in $K$, and that there is at least one such prime $q$ at which $E[p]$ is ramified,
\item[(iv)] $E(\bQ_p)[p]=0$.
\end{itemize}
Then $X_\pp$ is $\Lambda$-torsion and
\[
{\rm char}_{\Lambda}(X_\pp)=(L_\pp(f))
\]
as ideals in $\Lambda_{R_0}$.
\end{thm}

\begin{proof}
This is \cite[Thm.~4.4]{cas-CJM}, 
in the final form given in \cite[Thm.~1.1]{cas-CJM-err}.
\end{proof}

As is well-known in Iwasawa theory, given an ``explicit reciprocity law''  expressing $L_\pp(f)$ as the image under a Perrin-Riou big logarithm map of a $\Lambda$-adic system of cohomology classes $\mathfrak{Z}_\infty$, the main conjecture for $L_\pp(f)$ can be alternatively stated (in favorable  circumstances) in the form of a main conjecture without reference to $p$-adic $L$-functions (but rather $\mathfrak{Z}_\infty$ instead) in the spirit of 
\cite[Conj.~B]{PR-HP}. Here we explain the form this takes in our setting.

Put $\mathscr{F}^+E[p^\infty]:=\mathscr{F}^+T\otimes\bQ_p/\bZ_p$, which defines a $G_{\bQ_p}$-invariant submodule of $E[p^\infty]$ under the isomorphism $E[p^\infty]\simeq T\otimes_{}\bQ_p/\bZ_p$. 
Following \cite{greenberg-iwasawa,greenberg-zeros}, define 
\[
\mathfrak{Sel}_{\rm Gr}(K_\infty,E[p^\infty]):={\rm ker}\biggl\{{\rm H}^1(K_\infty,E[p^\infty])\rightarrow\prod_{w\nmid p}{\rm H}^1(K_{\infty,w},E[p^\infty])\times\prod_{w\mid p}\frac{{\rm H}^1(K_{\infty,w},E[p^\infty])}{{\rm H}_{\rm Gr}^1(K_{\infty,w},E[p^\infty])}\biggr\},
\]
where for $w\mid p$ we put 
\[
{\rm H}^1_{\rm Gr}(K_{\infty,w},E[p^\infty]):={\rm ker}\bigl\{{\rm H}^1(K_{\infty,w},E[p^\infty])\rightarrow{\rm H}^1(I_w,E[p^\infty]/\mathscr{F}^+E[p^\infty])\bigr\}
\] 
with $I_w\subset G_{K_{\infty,w}}$ a decomposition group at $w$; and define the \emph{strict Selmer group} by 
\[
{\rm Sel}_{\rm str}(K_\infty,E[p^\infty]):={\rm ker}\biggl\{{\rm H}^1(K_\infty,E[p^\infty])\rightarrow\prod_{w\nmid p}{\rm H}^1(K_{\infty,w},E[p^\infty])\times\prod_{w\mid p}\frac{{\rm H}^1(K_{\infty,w},E[p^\infty])}{{\rm H}_{\rm str}^1(K_{\infty,w},E[p^\infty])}\biggr\},
\]
where for $w\mid p$ we put 
\[
{\rm H}^1_{\rm str}(K_{\infty,w},E[p^\infty]):={\rm ker}\{{\rm H}^1(K_{\infty,w},E[p^\infty])\rightarrow{\rm H}^1(K_{\infty,w},E[p^\infty]/\mathscr{F}^+E[p^\infty])\}.
\] 
Then clearly ${\rm Sel}_{\rm str}(K_{\infty},E[p^\infty])\subset\mathfrak{Sel}_{\rm Gr}(K_\infty,E[p^\infty])$, and there is a $\Lambda$-module pseudo-isomorphism 
\[
{\rm Sel}_{p^\infty}(E/K_\infty)\sim{\rm Sel}_{\rm str}(K_\infty,E[p^\infty])
\]
(see e.g. \cite{coates-greenberg}).  
On the other hand, let $\check{S}_\pp$ denote the subgroup of 
\[
{\rm H}^1(K,\mathbf{T})\simeq\varprojlim_L{\rm H}^1(L,T)
\] 
where $L$ runs over the finite extensions of $K$ contained in $K_\infty$, cut out by the local conditions that are everywhere orthogonal complements to ${\rm H}_\pp^1(K_{\infty,w},E[p^\infty])$ under local Tate duality. 
\begin{prop}\label{prop:equiv}
Suppose $\mathbf{z}_\infty^{*}$ is not $\Lambda$-torsion and the localization map
\[
{\rm loc}_\pp:\check{S}_p\rightarrow{\rm H}^1(K_\pp,\mathscr{F}^+\mathbf{T})
\]
is nonzero. Then the following are equivalent:
\begin{itemize}
\item[(i)] $X_\pp$ and $\check{S}_\pp$ are both   $\Lambda$-torsion, with
\[
{\rm char}_\Lambda(X_\pp)\Lambda_{R_0}=\bigl(L_\pp\bigr(f)).
\]
\item[(ii)] $X$ and $\check{S}_p$ both have $\Lambda$-rank one, with
\[
{\rm char}_\Lambda(X_{\rm tors})={\rm char}_\Lambda\bigl(\check{S}_p/\Lambda\mathbf{z}_\infty^{*}\bigr)^2.
\]
\end{itemize}
\end{prop}

\begin{proof}
As in the good $p$-ordinary case considered in \cite[Thm.~5.1]{BCK}, 
this can be readily extracted from the arguments in \cite[Appendix~A]{cas-BF}, replacing the appeal to the explicit reciprocity law of \cite[App.~A.1]{cas-BF} (a  special case of \cite[Thm.~5.7]{cas-hsieh1}) by Theorem~\ref{thm:ERL} and Corollary~\ref{cor:ERL} above. 
\end{proof}

\begin{rem}
The assertions in part (ii) of Proposition~\ref{prop:equiv} are a natural extension of Perrin-Riou's Heegner point main conjecture \cite[Conj.~B]{PR-HP} to multiplicative primes $p$. Moreover, in the \emph{split} multiplicative case, they might be viewed as a ``primitive'' of the following:
\begin{conj}\label{conj:HPMC-mult}
$\mathfrak{X}$ and $\check{S}_p$ both have $\Lambda$-rank one, with
\[
{\rm char}_\Lambda(\mathfrak{X}_{\rm tors})={\rm char}_\Lambda\bigl(\check{S}_p/\Lambda\mathbf{z}_\infty^{}\bigr)^2.
\]
\end{conj}
\noindent Indeed, directly from the definitions we see that  $\mathfrak{Sel}_{\rm Gr}(K_\infty,E[p^\infty])/{\rm Sel}_{\rm str}(K_\infty,E[p^\infty])$ injects into
\begin{align*}
&\prod_{w\mid p}{\rm ker}\bigl\{{\rm H}^1(K_{\infty,w},E[p^\infty]/\mathscr{F}^+E[p^\infty])\rightarrow{\rm H}^1(I_{w},E[p^\infty]/\mathscr{F}^+E[p^\infty])\bigr\}\\
&\quad\simeq\prod_{w\mid p}{\rm H}^1(G_{K_{\infty,w}}/I_w,(E[p^\infty]/\mathscr{F}^+E[p^\infty])^{I_w})\simeq\prod_{v\in\{\pp,\overline{\pp}\}}(\bQ_p/\bZ_p)/(\alpha-1)(\bQ_p/\bZ_p)
\end{align*}
(cf. \cite[\S{3.2}]{skinner-mult}). Using this, it should be possible to show that
\[
{\rm char}_\Lambda(\mathfrak{X}_{\rm tors})=
\begin{cases}
{\rm char}_\Lambda(X_{\rm tors})\cdot(\gamma-1)^2&\textrm{if $E$ has split multiplicative reduction at $p$,}\\[0.2em]
{\rm char}_\Lambda(X_{\rm tors})&\textrm{otherwise.}
\end{cases}
\]
On the other hand, since clearly
\[
{\rm char}_\Lambda\bigl(\check{S}_p/\Lambda\mathbf{z}_\infty^{}\bigr)={\rm char}_\Lambda\bigl(\check{S}_p/\Lambda\mathbf{z}_\infty'\bigr)\cdot(\gamma-1),
\]
we see that the assertions in part (ii) of Proposition~\ref{prop:equiv} should be equivalent to Conjecture~\ref{conj:HPMC-mult} in the nonsplit multiplicative case, and to a ``primitive'' version of the same conjecture -- with the trivial zeros accounted for by the factor $(\gamma-1)^2$ removed -- in the split multiplicative case.
\end{rem}

\subsection{Mazur's control theorem}

For every $n$, put $\Gamma_n$ for the unique subgroup of $\Gamma$ of index $p^n$, and set $K_n=(K_\infty)^{\Gamma_n}$, so ${\rm Gal}(K_n/K)=\Gamma/\Gamma_n\simeq\bZ/p^n\bZ$.

\begin{thm}\label{thm:control}
The natural restriction map
\[
{\rm Sel}_{p^\infty}(E/K_n)\rightarrow{\rm Sel}_{p^\infty}(E/K_\infty)^{\Gamma_n}
\]
has finite kernel and cokernel, of order bounded independently of $n$.
\end{thm}

\begin{proof}
Since we assume that $p$ splits in $K$, this follows from \cite[Prop.~3.7]{greenberg-cetraro} and the main result of \cite{BSDGG} (showing that ${\rm log}_p(q_E)\neq 0$ for the Tate period $q_E$ of $E/\bQ_p$).
\end{proof}

\section{Proof of main results}


\begin{proof}[Proof of Theorem~\ref{thm:HPMC}]
It follows from Theorem~1.10 in \cite{CV-dur} and the definition of $\mathbf{z}_\infty'$ that the $\Lambda$-adic Heegner class $\mathbf{z}_\infty^*$ is non-torsion (both in the split and the nonsplit multiplicative case), 
%
and by the same argument as in \cite[Cor.~4.5]{BCK}, this yields ${\rm loc}_\pp(\check{S}_p)\neq 0$. In view of Proposition~\ref{prop:equiv}, the result now follows from Theorem~\ref{thm:BDP-IMC}.
\end{proof}

\begin{proof}[Proof of Theorem~\ref{thm:p-conv}]

Choose an imaginary quadratic field $K$ such that:
\begin{itemize}
\item[(i)] $q$ is ramified in $K$,
\item[(ii)] if $N$ is odd or divisible by $4$, then $2$ splits in $K$,
\item[(iii)] every prime factor $\ell\neq q$ of $N$ splits in $K$,
\item[(iv)] $L(E^K,1)\neq 0$.
\end{itemize}
Note that condition (iii) implies in particular that $p$ splits in $K$. The existence of such $K$ (in fact, of infinitely many such $K$) is ensured by \cite[Thm.~B.1]{friedberg-hoffstein}. 

By condition (iv) and Kolyvagin's (or alternatively, Kato's) work, we have $\#{\rm Sel}_{p^\infty}(E^K/\bQ)<\infty$, and so our corank one assumption implies that 
\begin{equation}\label{eq:cork-1}
{\rm corank}_{\bZ_p}{\rm Sel}_{p^\infty}(E/K)=1.
\end{equation}
Let $\gamma\in\Gamma$ be a topological generator.  
By Theorem~\ref{thm:control}, it follows that ${\rm rank}_{\bZ_p}(X/(\gamma-1)X)=1$, and so the first assertion of Theorem~\ref{thm:HPMC} implies that 
\[
{\rm char}_\Lambda(X_{\rm tors})\not\subset (\gamma-1).
\]
By the equality of characteristic ideals in the same result, this shows that
\[
{\rm char}_{\Lambda}(\check{S}_p/\Lambda\mathbf{z}_\infty^{*})\not\subset(\gamma-1),
\]
and so (by another application of Theorem~\ref{thm:control}, and noting that ${\rm rank}_{\bZ_p}{\rm Sel}(K,T)=1$ by \eqref{eq:cork-1}) $\mathbf{z}_\infty^{*}$ has non-torsion image under the projection
\[
{\rm pr}_0:\check{S}_p\rightarrow{\rm Sel}(K,T).
\]
By Theorem~\ref{thm:exc-zero} in the split multiplicative case, and its proof in the nonsplit multiplicative case (so $(1-\alpha^{-1})\neq 0$), ${\rm pr}_0(\mathbf{z}_\infty^{*})$ is a nonzero multiple of $\kappa_f$, and hence by the Gross--Zagier formula (see \eqref{eq:GZ-cor}) the above shows that ${\rm ord}_{s=1}L(E/K,s)=1$. Together with condition (iv) and the factorization $L(E/K,s)=L(E,s)L(E^K,s)$, this yields the result.
\end{proof}

\bibliographystyle{amsalpha}
\bibliography{Birch-conv-refs}

\end{document}